\documentclass[11pt]{article}
\usepackage{amsmath, amssymb, amsthm}
\usepackage{graphicx}
\usepackage{sham}
\usepackage{mathrsfs}
\usepackage{calc,tikz}
\usepackage{float}
\usetikzlibrary{arrows,decorations.markings,positioning,shapes.geometric}

\tikzstyle{vertex}=[circle, draw, inner sep=1pt, minimum size=4pt]

\tikzstyle{ann} = [fill=white,font=\footnotesize,inner sep=1pt]
\tikzstyle{arrow} = [thick,<-->,>=stealth]

\DeclareMathAlphabet{\mathpzc}{OT1}{pzc}{m}{it}

\title{\Large\bf Construction of additively graceful signed graphs-I}
\author{Mukti Acharya}

\begin{document}
\maketitle
\address{Former Professor and Head, Department of Applied Mathematics \\ Delhi Technological University, Delhi- 110042\\
mukti1948@gmail.com}  

\begin{abstract}	
%{\footnotesize }
In this paper, we construct additively graceful signed graphs $S$ from a given graph $G$ that may be additively graceful or not be additively graceful. We also show the construction of additively graceful signed graphs from additively graceful signed graphs. We find the values of $m,n$ in non-divisible sum graph, denoted as $G(m,n)$, that admit additively graceful labeling. 

{\bf Keywords}: Graceful graph, additively graceful graph, signed graph, additively graceful signed graph, non-divisible sum graph.

\indent {\small {\bf 2020 Mathematics Subject Classification codes}}:05C78
\end{abstract}

 \hrulefill

%{\small \textbf{Keywords: .} }

%\indent {\small {\bf 2020 Mathematics Subject Classification codes: 05C07, 05C45, 05C69, 05C99} }

\section{Introduction}
By a graph, we mean an undirected  simple graph. All graphs considered in this paper are connected and finite unless otherwise mentioned. Let $G = (V, E)$ be a graph such that $|V (G)|$ = $p$ and $|E(G|$ = $q$, where $V(G)$ and $E(G)$ are called the vertices and edges in G respectively. Consider a signed graph $S=(G,\sigma)$, where $G=(V,E)$ is called the underlying graph of $S$ and $\sigma: E(G)\rightarrow \{+,-\}$ is called the signature of $G$. The edges receiving the  $+$ sign are called positive edges, and the edges receiving the $-$ sign are called negative edges. Positive edges are drawn as solid lines  and negative edges are drawn as dashed lines as shown in Figure 1. If $\sigma :uv \rightarrow \{+\}$ for every edge $uv$ in $S$ then $S$ is called an all-positive signed graph and if $\sigma :uv \rightarrow \{-\}$ for every edge $uv$ in $S$ then $S$ is called an all-negative  signed graph.  A signed graph is called homogeneous if it is either all-positive or  all-negative and is  heterogeneous otherwise. By a positive homogeneous signed graph we mean an all-positive signed graph and a negative homogeneous signed graph is  an all-negative  signed graph. 
 By a $(p,m,n)$ signed graph $S$, we mean a signed graph $S$ such that $|V (S)|$ = $p$ and $|E(S|$ = $m+n$. By $E^+(S)$ and $E^-(S)$ we denote the set of positive and negative edges in $S$, respectively. In our case $|E^+(S)|$ = $m$ and $|E^-(S)|$ = $n$. \\

Today there is a wealth of literature on various types of graph labeling, and all of them can not be listed here, but to have some information on this topic, we refer to \cite{G}. For all the terms in graph theory not defined here, we refer to \cite{H} and for signed graph labeling, we refer to (\cite{MA1},\cite{ACK}) \\
Now we give a few definitions that are needed for our work.
\begin{defn}\cite{G}   Let G be a $(p,q)$ graph and for an injective function $f : V (G)\to\{0, 1,  \ldots,q\}$  when  each edge $uv$ in $G$ is given the  label  $ g(uv) = |f(u) - f(v)|$  such that the resulting edge labels are distinct  and range from $1,2,..., q$. A graph $G$ that admits such  labeling is called a graceful graph. 
\end{defn}
\begin{defn} \cite{MA}
 Let $S=(p,m,n)$ be a signed graph and $f : V (S)\to\{0, 1,  \ldots,q=m+n\}$ be an injective function and when  each edge $uv$ in $S$ is given the  label  $ g(uv) =\sigma(uv) |f(u) - f(v)|$  such that the positive edges receive the labels 1,2,..., m and the negative edges receive the labels  -1,-2,...,-n. A signed graph $S$ that admits such  labeling is called a graceful signed graph.
\end{defn}

\begin{defn} \cite{S}  A $(p,q)$-graph $G=(V,E)$  with $q\geq 1$ and $p\geq 2$ is said to be additively graceful if it admits a labeling $f : V (G)\to\{0, 1, ...,\left\lceil\frac{q+1}{2}\right\rceil\}$ such that the edge induced labels defined as $g(uv)$ = $f(u)+ f(v)$ are all distinct and range from $1,2,..., q$. 
\end{defn}
The following theorem gives the relation between the number of vertices and the number of edges if  a graph $G$ is additively graceful. 
\begin{thm} \cite{S} If $G$ is additively graceful $(p,q)-graph$ then $q\geq 2p-4$ and the bounds are best possible. 
\end{thm}

Motivated by the concept of additively graceful graphs, Pereira et al. \cite{J}  have initiated the study on additively graceful signed graphs. For the  definition of additively graceful signed graphs, we refer to D’Souza and Pereira \cite{B}.
\begin{defn} For a $ (p,m,n)$ signed graph $S$, an additively graceful
labeling of $S$ is an injective mapping $f: V(S)\to\{0, 1, \ldots, (m + \left\lceil\frac{n+1}{2}\right\rceil)\}$  such that the  edge function defined as $g(uv) = f(u) + f(v)$ for every $uv$ in $E^-( S)$ and $g(uv) =|f(u) - f(v)|$ for every   $uv$ in  $E^+(S)$ is such that $E^+(S)= \{1,2,...,m\}$ and $E^-(S)= \{1,2, \ldots,n\}$. A signed graph which admits such a labeling is called an additively graceful signed graph. 
\end{defn}
\begin{defn}\cite{IC} The vertex set of $G_{m,n}$ is $V = \{1, 2, 3, \ldots, n\}$ and two distinct vertices $a, b \in V( G_{m,n})$ are adjacent if and only if $a \neq b$ and $a + b$ is not divisible by $m$, where $m \in \mathbb{N}$ and $m > 1$.
\end{defn}
\begin{defn} \cite{CLB}
A graph is called complement reducible if it can be reduced to an edgeless graph by successively taking complements within components. 
\end{defn}

\section{ Construction of Additively Graceful Signed Graphs} 
\begin{thm} Let $(u,v,w)$  be an all-negative $P_3$. We take $m$ pendant vertices and join them to $w$ by positive edges, and such a constructed signed graph $S$ is an additively graceful signed graph.
\end{thm}
\begin{proof} Let $p_1,p_2,\ldots, p_m$ be  pendant vertices that are joined to $w$ by positive edges. Clearly $|V(S)|= m+3$ and $|E(S)|= m+2$, where $m$ is the number of positive edges. We claim that $S$ is an additively graceful signed graphs. For $f: V(S)\rightarrow\{0,1,2,\ldots,(m+2)\}$ we define $f(v)= 0$, $f(u)= 1$, $f(w)= 2$ and $f(p_i)=2+i,i\rightarrow\{1,2,..., m\}$.  
 The negative edges $uv$ and $vw$ receive  labels 1 and 2 respectively, and the positive edges  $wp_i, i=1,2,...,m $ receive  labels $1,2,...,m $, respectively. The construction of this signed graph is shown in Figure 1.
\end{proof}

\begin{figure} [H]
\centering
\begin{tikzpicture} [scale=1,auto=left]
\node[vertex] (v1) at (1,0) [fill, label=below : $0$] {};
\node[vertex] (v7) at (2.5,0)  [fill,label=below:$1$] {};
\node[vertex] (v11) at (-0.5,0) [fill,label=below:$2$]  {};

\node[vertex] (v4) at (-1.0,1) [fill,label=above:$3$] {};

\node[vertex] (v5) at (-2.5,1) [fill,label=above:$4$] {};
\node[vertex] (v10) at (-2.5,-1) [fill,label=below:$5$] {};
\node[vertex] (v6) at (-1.0,-1) [fill, label=below : $6$] {};

\path
(v11) edge [dashed] (v1)
(v11) edge (v10)
(v11) edge (v4)
(v1)  edge [dashed] (v7)
(v11) edge (v5)
(v11) edge (v6)
;
\end{tikzpicture}
\caption{}
\end{figure} 

 Note that $P_3$ which we have  taken to construct an additively graceful signed graph is an additively graceful signed graph, and the underlying graph of $P_3$ is also additively graceful. The constructed S is a bistar. We pose the following question and try to find its answer.

Question. Is it the only additively graceful signed bistar?

To answer this question, first we prove the following result.
\begin{thm}  $ T= K_1,_p$ is an additively graceful signed graph if and only if it
 has exactly one negative edge.
\end{thm}
%\begin{figure} [H]
\begin{proof} Necessity: $|V(T)| = m+n+1$, where  $m$ are positive edges and $n$ are negative edges, and $u$ is a vertex in $T$ such that  $d(u)= p=m+n$. By $v_i, i\rightarrow\{1,2,..., m+n\}$ we mean the pendent vertices. 
Let $T$ be additively graceful, that is, it admits additively graceful labeling and we assume that $n\geq 2$. Without loss of generality, we assume that $n = 2$ and $f : V (G)\to\{0, 1,\ldots,(m+2)\}$ is an injective function. According to the definition of an additively graceful signed graph, the labels of negative edges must be 1 and 2. Now we assign labels $0,1$ and $2$ to $u,v_1$ and $v_2$ respectively, where $(u,v_1)$ and $(u, v_2)$ are the negative edges. 
The labels of positive edges must be $(1,2,3,\ldots,m)$ which is not possible, as to obtain the edge labels 1 and 2 we have to label the end vertices of two positive edges by 1 and 2. Thus, it is a contradiction to the fact that $f : V (G)\to\{0, 1,\ldots,(m+2)\}$ is an injective function. Hence, $n\geq 2$ does not hold and $T$ must have exactly one negative edge.\\
Sufficiency: Now we show that $T$ admits additively graceful labeling when it has exactly one negative edge. For $f : V (G)\to\{0, 1,\ldots,(m+1)\}$, we define 
 $f(u) = 1$, $f(v_1) = 0$ and $f(v_i)= f(u)+i ,i = 1,2,..., m$, where $v_i's$ are the vertices adjacent to $u$ by positive edges. Clearly, the negative edge $uv_1$ has label $1$, and  positive edges  $uv_i,i= 2,3,...,m$ have labels  {1,2,...,m}. Hence, $T$ having exactly $1$ negative edge is an additively graceful signed graph.
\end{proof} 
Now we construct a bistar $B(m,n)$  obtained from a positive edge $uv$ having $m-1$ positive edges incident with $u$ and a negative edge $vw$ incident with $v$. 
\begin{thm}  $B(m,n)$  obtained from a positive edge $uv$ having $(m-1)$ positive edges incident with $u$ and one negative edge $vw$ incident with $v$ admits additively graceful labeling.
\end{thm}
\begin{proof}  In $B(m,n)$, $|V(B(m,n)|=m+2$. For $f : V(B(m,n))\to\{0, 1, , \ldots,(m+1)\}$, we define \\  
$f(v) = 1$,  $f(w) = 0$ , $f(u)= m+1$ and $f(w_i)= f(u)-i ,i\rightarrow\{1,2,\ldots, (m-1)\}$, where $w_i's$ are the vertices adjacent to $u$ by positive edges. Clearly, the negative edge $vw$ has label $1$, the positive edge  $uv$ has label m, and the positive edges $uw_i$,$i= 1,2,...,m$-$1$ have labels  {1,2,...,m$-$1}. Thus, $B(m,n)$ admits additively graceful labeling. An additively graceful labeling of $B(4,1)$ is shown in Figure 2.
\end {proof}

\begin{figure} [H]
\centering
\begin{tikzpicture} [scale=1,auto=left]
\node[vertex] (v1) at (1,0) [fill, label=below : $1$] {};
\node[vertex] (v7) at (2.5,0)  [fill,label=below:$0$] {};
\node[vertex] (v11) at (-0.5,0) [fill,label=below:$5$]  {};

\node[vertex] (v4) at (-.71,2) [fill,label=above:$2$] {};

\node[vertex] (v5) at (-2.5,1) [fill,label=left:$3$] {};
\node[vertex] (v10) at (-2.5,-1) [fill,label=left:$4$] {};

\path
(v11) edge (v1)
(v11) edge (v10)
(v11) edge (v4)
(v1) edge [dashed] (v7)
(v11) edge (v5)
;
\end{tikzpicture}
\caption{} \label{f:fig2}
\end{figure} 

Observe that $B(m,n)$, obtained from an all-positive $K_1,_m$, is graceful as well as  additively graceful,  but its underlying  graph is not additively graceful.\\
\indent In \cite{S}, it has been shown that $K_3$ is an additively graceful graph. Using 
$K_3$, we construct an additively graceful signed graph. Let $ u,v,w $ be the vertices of a all-negative $K_3$ and at $w$ we attach $m$ positive edges and denote this signed graph by $ST$. \\
We also construct an additively graceful signed graph from $K_3$ having   exactly one negative edge and attach $m$ pendant vertices by positive edges at a vertex of $K_3$ having negative degree one and denote this signed graph by $STE$. Now, we show that $ST$ and  $STE$ admit additively graceful labeling.
\begin{thm} $ST$ admits an  additively graceful labeling.
\end{thm}
\begin{proof} Clearly $|V(ST)| = m+3$ and $|E(ST)|= m+3$. For $f : V(ST)\to\{0, 1, \ldots,(m+2)\}$, we define \\  
$f(u)= 0$, $f(v)= 1$, $f(w)= 2$  and $f(w_i)= (f(w)+i), i\rightarrow\{1,2, \ldots, (m +2)\}$, where $w_i's$ are the vertices joined to $w$ by positive edges. Clearly, the labels of the negative edges $uv$, $uw$ and $vw$  are $1$,$2$ and $3$ respectively, and the labels of the positive edges $ww_i, i=1,2,...,m$ are  $1,2,...,m$, respectively. Thus, $ST$ admits additively graceful labeling, and an additively graceful labeling  of $ST$ is shown in Figure 3.  
\end{proof} 

\begin{figure} [H]
\centering
\begin{tikzpicture} [scale= 1,auto=left]
\node[vertex] (v1) at (1,0) [fill, label=below : $6$] {};
\node[vertex] (v7) at (1.0,1)  [fill,label=below:$5$] {};
\node[vertex] (v11) at (-0.5,0) [fill,label=below:$2$]  {};
\node[vertex] (v4) at (-1.0,1) [fill,label=above:$4$] {};
\node[vertex] (v5) at (-2.5,1) [fill,label=left:$3$] {};
\node[vertex] (v10) at (-2.5,-1) [fill,label=left:$0$] {};
\node[vertex] (v6) at (-1.0,-1) [fill, label=below : $1$] {};

\path
(v11) edge  (v1)
(v11) edge [dashed] (v10)
(v10) edge [dashed] (v6)
(v11)  edge (v7)
(v11) edge (v5)
(v11) edge [dashed] (v6)
(v11) edge (v4)
(v11) edge (v1)
;
\end{tikzpicture}
\caption{}
\end{figure} 

Remark: It can be verified that $ST$ also admits graceful labeling, that is, it is a graceful signed graph. The underlying graph of $ST$ is graceful, but not additively graceful. 
\begin{thm} $STE$ admits an additively graceful labeling.   
\end{thm}
\begin{proof} Clearly $|V(STE)| = m+3$ and $|E(STE)|= m+3$. For $f : V(STE)\to\{0, 1,  \ldots,(m+3)\}$, we define \\  
$f(u) = 1$, $f(v) = 2$,  $f(w) = 0$, and $f(w_i)= f(w)+(i+2) ,i\rightarrow\{1,2, \ldots, (m +2)\}$, where $w_i's$ are joined   to $w$ by the positive edges and $d^-(u)= d^-{w} = 1$ . Clearly, the label of  a negative edge $uw$ of $STE$ is $1$ and the labels of positive edges $uv$,$vw$ and   $w_i, i\rightarrow\{1,2, \ldots, (m +2)\}$,   are  $1,2, \ldots,(m+2)$, respectively. Thus, $STE$ admits additively graceful labeling, and hence is additively graceful. An additively graceful labeling of $STE$ is shown in Figure 4.  
\end {proof}

\begin{figure} [H]
\centering
\begin{tikzpicture} [scale=1,auto=left]
\node[vertex] (v1) at (1,0) [fill, label=below : $6$] {};
\node[vertex] (v7) at (1.0,1)  [fill,label=below:$5$] {};
\node[vertex] (v11) at (-0.5,0) [fill,label=below:$0$]  {};
\node[vertex] (v4) at (-1.0,1) [fill,label=above:$4$] {};
\node[vertex] (v5) at (-2.5,1) [fill,label=left:$3$] {};
\node[vertex] (v10) at (-2.5,-1) [fill,label=left:$2$] {};
\node[vertex] (v6) at (-1.0,-1) [fill, label=below : $1$] {};

\path
(v11) edge (v1)
(v11) edge (v10)
(v10) edge (v6)
(v11)  edge (v7)
(v11) edge (v5)
(v11) edge [dashed] (v6)
(v11) edge (v4)
(v11) edge (v1)
;
\end{tikzpicture}
\caption{}
\end{figure}

It is easy to verify that the underlying graph of $STE$ admits graceful labeling.
In Figure 5, we show  additively graceful labeling for signed graphs that have  a triangle as the induced subgraph.  Verify that the graph obtained by joining $x$ pendent vertices to a vertex of a triangle is graceful.

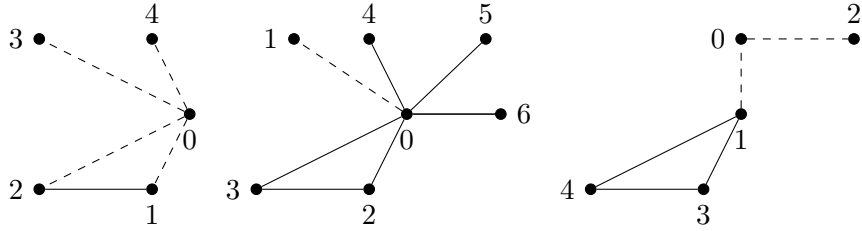
\begin{figure} [H]
\centering
\begin{tikzpicture} [scale= 1,auto=left]
\node[vertex] (v11) at (-0.5,0) [fill,label=below:$0$]  {};
\node[vertex] (v4) at (-1.0,1) [fill,label=above:$4$] {};
\node[vertex] (v5) at (-2.5,1) [fill,label=left:$3$] {};
\node[vertex] (v10) at (-2.5,-1) [fill,label=left:$2$] {};
\node[vertex] (v6) at (-1.0,-1) [fill, label=below : $1$] {};

\path
(v11) edge [dashed] (v10)
(v10) edge (v6)
(v11) edge [dashed] (v5)
(v11) edge [dashed] (v6)
(v11) edge [dashed] (v4)

;
\end{tikzpicture}
\begin{tikzpicture} [scale= 1,auto=left]
\node[vertex] (v1) at (0.75,0) [fill, label=right: $6$] {};
\node[vertex] (v7) at (0.550,1)  [fill,label=above:$5$] {};
\node[vertex] (v11) at (-0.5,0) [fill,label=below:$0$]  {};
\node[vertex] (v4) at (-1.0,1) [fill,label=above:$4$] {};

\node[vertex] (v5) at (-2.0,1) [fill,label=left:$1$] {};
\node[vertex] (v10) at (-2.5,-1) [fill,label=left:$3$] {};
\node[vertex] (v6) at (-1.0,-1) [fill, label=below : $2$] {};

\path
(v11) edge (v1)
(v11) edge (v10)
(v10) edge (v6)
(v11)  edge (v7)
(v11) edge [dashed] (v5)
(v11) edge  (v6)
(v11) edge (v4)
(v11) edge (v1)
;
\end{tikzpicture}
\begin{tikzpicture} [scale= 1,auto=left]
%\node[vertex] (v1) at (0.75,0) [fill, label=below : $6$] {};
%\node[vertex] (v7) at (0.550,1)  [fill,label=below:$5$] {};
\node[vertex] (v11) at (-0.5,0) [fill,label=below:$1$]  {};
\node[vertex] (v4) at (1.0,1) [fill,label=above:$2$] {};

\node[vertex] (v5) at (-0.50,1) [fill,label=left:$0$] {};
\node[vertex] (v10) at (-2.5,-1) [fill,label=left:$4$] {};
\node[vertex] (v6) at (-1.0,-1) [fill, label=below : $3$] {};

\path

(v11) edge (v10)
(v10) edge (v6)
(v11) edge [dashed] (v5)
(v11) edge  (v6)
(v5) edge [dashed] (v4)

;
\end{tikzpicture}
\caption{Left: a \, Middle: b  Right: c}
\end{figure} 
Verify that the graph obtained by joining $x$ pendent vertices to a vertex of a triangle is graceful.\\
\indent Now we give the construction of an  additively graceful signed graph obtained from negative homogeneous $K_4$ by labeling its vertices with  $u,v,w,x$ and joining vertex $x$ to $m$ vertices by positive edges.
\begin{thm} A signed graph $S$ obtained from a negative homogeneous $K_4$ by joining a vertex of $K_4$ to m pendant vertices by positive edges is an additively graceful signed graph.  
\end{thm}
\begin{proof} $|V(S)| = m+4$ and $|E(S)| =m+6$. Now we show that $S$ admits additively graceful labeling. For $f : V(S)\to\{0, 1, \ldots,(m+4)\}$, we define 
$f(u) = 0$, $f(v)=1$, $f(w)= 2$,$f(x)=4$, and $f(w_i)= f(x)+i ,i\rightarrow\{1,2, \ldots, m \}$, where $w_i's$ are the vertices incident with  $w$ by positive edges. Clearly, $1,2,3,4,5$ and $6$ are the labels of the negative edges $ uv, uw, vw, ux,vx$,and $ wx$, respectively. The labels of the positive edges $xw_i, i=1,2,...,m$ are $1,2,...,m$, respectively. Thus, $S$ admits additively graceful labeling. The  labels  of this newly constructed additively graceful signed graph are shown in Figure 6(a).
\end {proof}
Observe that the underlying graph of $S$ admits graceful labeling as shown in Figure 6(b).
\\

\begin{figure} [H]
\centering
\begin{tikzpicture} [scale=0.75,auto=left]
\node[vertex] (v1) at (0,-1.5) [fill, label=below : $4$] {};
\node[vertex] (v2) at (2.5,-0.5)  [fill,label=right:$8$] {};
\node[vertex] (v3) at (-0.5,0) [fill,label=above:$0$]  {};
\node[vertex] (v4) at (2.5,1) [fill,label=right:$6$] {};
\node[vertex] (v5) at (2.5,2) [fill,label=right:$5$] {};
\node[vertex] (v6) at (-3,-3) [fill,label=left:$1$] {};
\node[vertex] (v7) at (3.0,-3) [fill, label=right : $2$] {};
\node[vertex] (v8) at (2.5,-1.11) [fill, label=right : $9$] {};
\node[vertex] (v9) at (2.5,0.45) [fill,label=right:$7$] {};

\path
(v6) edge [dashed] (v7)
(v3) edge [dashed] (v6)
(v3) edge [dashed] (v7)
(v1) edge (v4)
(v3) edge [dashed] (v1)
(v6) edge [dashed] (v1)
(v7) edge [dashed] (v1)
(v1) edge (v2)
(v1) edge (v5)
(v1) edge (v9)
(v1) edge (v8)
;
\end{tikzpicture}
\begin{tikzpicture} [scale= 0.75,auto=left]

\node[vertex] (v1) at (0,-1.5) [fill, label=below : $0$] {};
\node[vertex] (v2) at (2.5,-0.5)  [fill,label=right:$10$] {};
\node[vertex] (v3) at (-0.5,0) [fill,label=above:$1$]  {};
\node[vertex] (v4) at (2.5,1) [fill,label=right:$8$] {};
\node[vertex] (v5) at (2.5,2) [fill,label=right:$5$] {};
\node[vertex] (v6) at (-3,-3) [fill,label=left:$3$] {};
\node[vertex] (v7) at (3.0,-3) [fill, label=right : $7$] {};
\node[vertex] (v8) at (2.5,-1.11) [fill, label=right : $11$] {};
\node[vertex] (v9) at (2.5,0.45) [fill,label=right:$9$] {};

\path
(v6) edge (v7)
(v3) edge (v6)
(v3) edge (v7)
(v1) edge (v4)
(v3) edge  (v1)
(v6) edge (v1)
(v7) edge (v1)
(v1) edge (v2)
(v1) edge (v5)
(v1) edge (v9)
(v1) edge (v8)
;
\end{tikzpicture}
\caption{Left:  6 (a)\, Right: 6 (b)}
\end{figure}
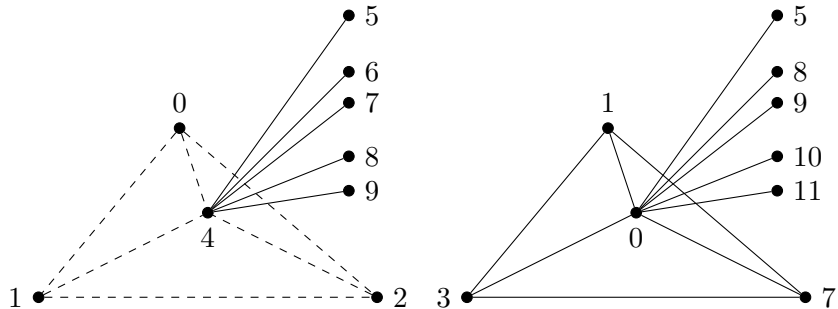

 Few  additively graceful signed graphs, whose underlying graph is isomorphic to $K_4$, have  been given in \cite{J}, and on them we construct additively graceful signed graphs as shown in Figure 7 and Figure 8.

\begin{figure} [H]
\centering
\begin{tikzpicture} [scale=0.65,auto=left]
\node[vertex] (v1) at (0,-1.5) [fill, label=below : $3$] {};
\node[vertex] (v2) at (2.5,-0.5)  [fill,label=right:$8$] {};
\node[vertex] (v3) at (-0.5,0) [fill,label=above:$1$]  {};
\node[vertex] (v4) at (2.5,1) [fill,label=right:$6$] {};
\node[vertex] (v5) at (2.5,2) [fill,label=right:$5$] {};
\node[vertex] (v6) at (-3,-3) [fill,label=left:$0$] {};
\node[vertex] (v7) at (3.0,-3) [fill, label=right : $2$] {};
\node[vertex] (v8) at (2.5,-1.11) [fill, label=right : $9$] {};
\node[vertex] (v9) at (2.5,0.45) [fill,label=right:$7$] {};

\path
(v6) edge [dashed] (v7)
(v3) edge [dashed] (v6)
(v3) edge (v7)
(v1) edge (v4)
(v3) edge [dashed] (v1)
(v6) edge [dashed] (v1)
(v7) edge [dashed] (v1)
(v1) edge (v2)
(v1) edge (v5)
(v1) edge (v9)
(v1) edge (v8)
;
\end{tikzpicture}
\begin{tikzpicture} [scale=0.65,auto=left]
\node[vertex] (v1) at (0,-1.5) [fill, label=below : $3$] {};
\node[vertex] (v2) at (2.5,-0.5)  [fill,label=right:$10$] {};
\node[vertex] (v3) at (-0.5,0) [fill,label=above:$0$]  {};
\node[vertex] (v4) at (-2.5,1) [fill,label=left:$6$] {};
\node[vertex] (v5) at (2.5,2) [fill,label=right:$7$] {};
\node[vertex] (v6) at (-3,-3) [fill,label=left:$1$] {};
\node[vertex] (v7) at (3.0,-3) [fill, label=right : $2$] {};
\node[vertex] (v8) at (-2.5,-0.5) [fill, label=left : $4$] {};
\node[vertex] (v11) at (-2.5,0.5) [fill, label=left : $5$] {};
\node[vertex] (v9) at (2.5,0.5) [fill,label=right:$9$] {};
\node[vertex] (v10) at (2.5,1.13) [fill,label=right:$8$] {};
\path
(v6) edge [dashed] (v7)
(v3) edge [dashed] (v6)
(v3) edge [dashed] (v7)
(v3) edge [dashed] (v4)
(v3) edge  (v1)
(v6) edge  (v1)
(v7) edge  (v1)
(v1) edge (v2)
(v1) edge (v5)
(v1) edge (v9)
(v3) edge [dashed] (v8)
(v3) edge [dashed] (v11)
(v1) edge (v10)
;
\end{tikzpicture}
\caption{Left:  (a), Right:  (b)}
\end{figure}

\begin{figure} [H]
\centering
\begin{tikzpicture} [scale=0.65,auto=left]
\node[vertex] (v1) at (0,-1.5) [fill, label=below : $4$] {};
\node[vertex] (v2) at (1.5,0)  [fill,label=above:$7$] {};
\node[vertex] (v3) at (-0.5,0) [fill,label=right:$0$]  {};
\node[vertex] (v4) at (-1.0,1) [fill,label=above:$3$] {};
\node[vertex] (v5) at (3.5,0) [fill,label=above:$8$] {};
\node[vertex] (v6) at (-3,-3) [fill,label=left:$1$] {};
\node[vertex] (v7) at (3.0,-3) [fill, label=right : $2$] {};

\path
(v6) edge  (v7)
(v3) edge [dashed] (v6)
(v3) edge [dashed] (v7)
(v3) edge [dashed] (v4)
(v3) edge (v1)
(v6) edge (v1)
(v7) edge (v1)
(v7) edge (v2)
(v7) edge (v5)
;
\end{tikzpicture}
\begin{tikzpicture} [scale=0.65,auto=left]
\node[vertex] (v1) at (0,-1.5) [fill, label=below : $5$] {};
\node[vertex] (v2) at (2.5,-0.5)  [fill,label=right:$10$] {};
\node[vertex] (v3) at (-0.5,0) [fill,label=above:$0$]  {};
\node[vertex] (v4) at (-2.5,1) [fill,label=above:$7$] {};
\node[vertex] (v5) at (2.5,2) [fill,label=right:$8$] {};
\node[vertex] (v6) at (-3,-3) [fill,label=left:$1$] {};
\node[vertex] (v7) at (3.0,-3) [fill, label=right : $2$] {};
\node[vertex] (v8) at (-2.5,-0.5) [fill, label=below: $6$] {};
\node[vertex] (v9) at (2.5,0.5) [fill,label=right:$9$] {};

\path
(v6) edge  (v7)
(v3) edge [dashed] (v6)
(v3) edge (v7)
(v3) edge (v4)
(v3) edge (v1)
(v6) edge (v1)
(v7) edge (v1)
(v3) edge (v2)
(v3) edge (v5)
(v3) edge (v9)
(v3) edge (v8)
;
\end{tikzpicture}
\caption{Left: (a), Right: (b)}
\end{figure} 

In Figure 9 it has been shown that the signed graph given in Figure 7(b) is a graceful signed graph. It is easy to see that the signed graphs as given in Figure 8 are graceful signed graphs. 

\begin{figure} [H]
\centering
\begin{tikzpicture} [scale=1,auto=left]
\node[vertex] (v1) at (0,-1.5) [fill, label=below : $3$] {};
\node[vertex] (v2) at (2.5,-0.5)  [fill,label=right:$10$] {};
\node[vertex] (v3) at (-0.5,0) [fill,label=above:$0$]  {};
\node[vertex] (v4) at (-2.5,1) [fill,label=left:$6$] {};
\node[vertex] (v5) at (2.5,2) [fill,label=right:$7$] {};
\node[vertex] (v6) at (-3,-3) [fill,label=left:$5$] {};
\node[vertex] (v7) at (3.0,-3) [fill, label=right : $2$] {};
\node[vertex] (v8) at (-2.5,-0.5) [fill, label=left : $1$] {};
\node[vertex] (v11) at (-2.5,0.5) [fill, label=left : $4$] {};
\node[vertex] (v9) at (2.5,0.5) [fill,label=right:$9$] {};
\node[vertex] (v10) at (2.5,1.3) [fill,label=right:$8$] {};
\path
(v6) edge [dashed] (v7)
(v3) edge [dashed] (v6)
(v3) edge [dashed] (v7)
(v3) edge [dashed] (v4)
(v3) edge  (v1)
(v6) edge  (v1)
(v7) edge  (v1)
(v1) edge (v2)
(v1) edge (v5)
(v1) edge (v9)
(v3) edge [dashed] (v8)
(v3) edge [dashed] (v11)
(v1) edge (v10)
;
\end{tikzpicture}
\caption{}
\end{figure}

\section{ Non-divisible Sum Graphs} 
In Figure 10, we have given some non-divisible sum  graphs $G_m,_n$  for given values of $m$ and $n$  which are additively graceful.  The labels of the vertices are from the set $\{0, 1, ...,\left\lceil\frac{q+1}{2}\right\rceil\}$, $q$ is in $G_m,_n$.
We observe that these graphs are complement reducible.

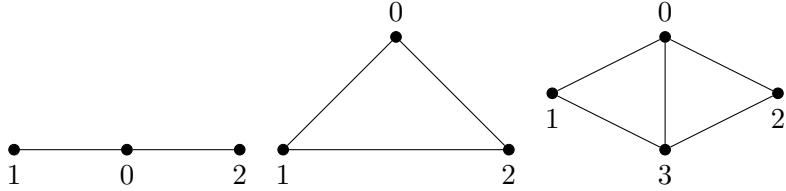
\begin{figure} [H]
\centering
\begin{tikzpicture} [scale=1.5,auto=left]
\node[vertex] (v1) at (0,0) [fill, label=below : $0$] {};
\node[vertex] (v2) at (-1,0)  [fill,label=below:$1$] {};
\node[vertex] (v3) at (1,0) [fill,label=below:$2$]  {};

\path
(v1) edge  (v2)
(v1) edge  (v3)
;
\end{tikzpicture}
\begin{tikzpicture} [scale=1.5,auto=left]
\node[vertex] (v1) at (0,0) [fill, label=above : $0$] {};
\node[vertex] (v2) at (-1,-1)  [fill,label=below:$1$] {};
\node[vertex] (v3) at (1,-1) [fill,label=below:$2$]  {};

\path
(v1) edge  (v2)
(v1) edge  (v3)
(v2) edge (v3)
;
\end{tikzpicture}
\begin{tikzpicture} [scale=1.5,auto=left]
\node[vertex] (v1) at (0,0) [fill, label=above:$0$] {};
\node[vertex] (v2) at (0,-1)  [fill,label=below:$3$] {};
\node[vertex] (v3) at (1,-.5) [fill,label=below:$2$]  {};
\node[vertex] (v4) at (-1,-0.5) [fill,label=below:$1$]  {};

\path
(v1) edge  (v2)
(v1) edge  (v3)
(v2) edge (v3)
(v1) edge (v4)
(v2) edge (v4)
;
\end{tikzpicture}
\caption{Left:$G_{2,3}$,\, Middle: $G_{6,3}$,\, Right:$G_{6,4}$}
\end{figure} 

The question that is naturally suggested by these results is whether there exist non-divisible sum  graphs $G_m,_n$ for given values of $m$ and $n$ that are additively graceful.
\section{ Further Scope}
In this paper, we have constructed some additively graceful signed graphs. Can the order and size of these signed graphs be increased?  More study is needed on  non-divisible sum  graphs, which admit   additively graceful labeling.\\  
 
Acknowledgements

 My sincere thanks to Dr. Ivy Chakrabarty, CHRIST ( Deemed to be University) Bangaluru( India)  for helping me in drawing the figures, without which it would not have been possible for me to complete this work.

\end{document}